\documentclass[psamsfonts]{amsart}

\usepackage{amssymb,amsfonts,amscd}
\usepackage[all,arc]{xy}
\usepackage{enumerate}
\usepackage{mathrsfs}

\newtheorem{thm}{Theorem}[section]
\newtheorem{cor}[thm]{Corollary}
\newtheorem{prop}[thm]{Proposition}
\newtheorem{lem}[thm]{Lemma}

\newtheorem*{thm1}{Theorem A}
\newtheorem*{thm4}{Theorem B}

\theoremstyle{definition}
\newtheorem{defn}[thm]{Definition}

\theoremstyle{remark}
\newtheorem{rem}[thm]{Remark}

\makeatletter
\let\c@equation\c@thm
\makeatother
\numberwithin{equation}{section}

\newcommand{\upperRomannumeral}[1]{\uppercase\expandafter{\romannumeral#1}}

\title[]{Nonnegative Ricci curvature, stability at infinity, and finite generation of fundamental groups}

\author[]{Jiayin Pan}

\newcommand{\Addresses}{{
		\bigskip
		\footnotesize
		
		Jiayin Pan, \textsc{Department of Mathematics, Rutgers University-New Brunswick, 110 Frelinghuysen Road, Piscataway, New Jersey 08854, USA.}\par\nopagebreak
		\textit{E-mail address}: \texttt{jp1016@math.rutgers.edu}

}}

\begin{document}

\begin{abstract}
   We study the fundamental group of an open $n$-manifold $M$ of nonnegative Ricci curvature. We show that if there is an integer $k$ such that any tangent cone at infinity of the Riemannian universal cover of $M$ is a metric cone, whose maximal Euclidean factor has dimension $k$, then $\pi_1(M)$ is finitely generated. In particular, this confirms the Milnor conjecture for a manifold whose universal cover has Euclidean volume growth and the unique tangent cone at infinity.
\end{abstract}

\maketitle

A longstanding problem in Riemannian geometry is the Milnor conjecture, which was proposed in 1968 \cite{Mi}. It states that any open $n$-manifold of nonnegative Ricci curvature has a finitely generated fundamental group. This conjecture remains open.

The Milnor conjecture has been verified under various additional assumptions. For a manifold with Euclidean volume growth, Anderson and Li independently have proven that the fundamental group is finite \cite{An,Li}. Sormani has showed that the Milnor conjecture holds if the manifold has small linear diameter growth, or linear volume growth \cite{Sor}. Liu has classified open $3$-manifolds of $\mathrm{Ric}\ge 0$, which confirms the Milnor conjecture in dimension $3$ \cite{Liu}. Recently, the author has presented a different approach to the Milnor conjecture in dimension $3$ \cite{Pan} with Cheeger-Colding theory \cite{CC1,CC2,CN1}.

In this paper, we verify the Milnor conjecture for manifolds with additional conditions on the Riemannian universal covers at infinity. For any open $n$-manifold $(M,x)$ of $\mathrm{Ric}\ge 0$, and any sequence $r_i\to \infty$, passing to a subsequence if necessary, we can consider a tangent cone of $M$ at infinity, which is the Gromov-Hausdorff limit \cite{Gro3} of
$$(r_i^{-1}M,x)\overset{GH}\longrightarrow(Y,y).$$
In general, a tangent cone of $M$ at infinity may not be unique \cite{CC2}. By splitting theorem \cite{CC1}, $Y$ is a metric product $\mathbb{R}^k\times Y'$, where $Y'$ has no lines. Cheeger and Colding have showed that when ${M}$ has Euclidean volume growth, any tangent cone of ${M}$ at infinity is a metric cone $(\mathbb{R}^k\times C(Z),(0,z))$ of dimension $n$ \cite{CC1}, where $C(Z)$ has $\mathrm{diam}(Z)<\pi$ and the vertex $z$. However, $k$ may not be unique among all tangent cones of $M$ at infinity \cite{CN2}.

We state our main result.
\begin{thm1}\label{main}
	Let $M$ be an open $n$-manifold of $\mathrm{Ric}\ge 0$. If there is an integer $k$ such that any tangent cone at infinity of the Riemannian universal cover of $M$ is a metric cone, whose maximal Euclidean factor has dimension $k$, then $\pi_1(M)$ is finitely generated.
\end{thm1}

If $k=0$, then in fact $\pi_1(M)$ is finite (Proposition \ref{0_sym}).

\begin{cor}\label{n_unique_up}
	Let $M$ be an open $n$-manifold of $\mathrm{Ric}\ge 0$. If the Riemannian universal cover of $M$ has Euclidean volume growth and the unique tangent cone at infinity, then $\pi_1(M)$ is finitely generated.
\end{cor}

Note that under the assumption in Theorem A, the Riemannian universal cover $\widetilde{M}$ of $M$ may have different tangent cones at infinity, even with different dimensions. We mention that the condition in Theorem A can be further weakened; see Remark \ref{rem_fix}.

For convenience, we introduce the following notion for the stability assumption in Theorem A.

\begin{defn}\label{def_k_sym}
	Let $M$ be an open $n$-manifold with $\mathrm{Ric}\ge 0$, and let $k$ be an integer. We say that $M$ is \textit{$k$-Euclidean at infinity}, if any tangent cone of $M$ at infinity $(Y,y)$ is a metric cone, whose maximal Euclidean factor has dimension $k$, that is, $(Y,y)$ splits as $(\mathbb{R}^k\times C(Z),(0,z))$, where $C(Z)$ is a metric cone with $\mathrm{diam}(Z)<\pi$ and the vertex $z$.
\end{defn}

We point out that, Definition \ref{def_k_sym} implies a uniform control on the diameter of $Z$: for any tangent cone of $M$ at infinity $(\mathbb{R}^k\times C(Z),(0,z))$, $\mathrm{diam}(Z)\le \pi-\eta(M)$ for some positive constant $\eta(M)$ (see Lemma \ref{gap_higher_sym} and Remark \ref{rem_gap_diam}).

We explain our approach to Theorem A as follows. Let $\widetilde{M}$ be the Riemannian cover of $M$. Given any tangent cone of $\widetilde{M}$ at infinity, we consider the associated equivariant Gromov-Hausdorff convergence with $\pi_1(M,x)$-action \cite{FY},
$$(r_i^{-1}\widetilde{M},\tilde{x},\pi_1(M,x))\overset{GH}\longrightarrow(\widetilde{Y},\tilde{y},G),$$
where $G$ is a closed subgroup of $\mathrm{Isom}(\widetilde{Y})$. We call such a limit space $(\widetilde{Y},\tilde{y},G)$ an equivariant tangent cone of $(\widetilde{M},\pi_1(M,x))$ at infinity. It is known that $G$ is always a Lie group \cite{CC3,CN1}. Recall that if $\widetilde{Y}$ is a metric cone $\mathbb{R}^k\times C(Z)$, where $\mathrm{diam}(Z)<\pi$, then $\mathrm{Isom}(\widetilde{Y})$ is a product $\mathrm{Isom}(\mathbb{R}^k)\times \mathrm{Isom}(Z)$. Thus $G$-action on $\widetilde{Y}$ can be naturally projected to $\mathbb{R}^k$-factor through $p: G\to \mathrm{Isom}(\mathbb{R}^k)$ (see Propositions \ref{cone_spit} and \ref{isom_gp_split}).

Our main discovery is that,
when $\widetilde{M}$ is $k$-Euclidean at infinity, there is certain equivariant stability among all the equivariant tangent cones of $(\widetilde{M},\pi_1(M,x))$ at infinity:

\begin{thm4}\label{main_omega}
	Let $M$ be an open $n$-manifold of $\mathrm{Ric}\ge 0$. Suppose that $\pi_1(M)$ is abelian and $\widetilde{M}$ is $k$-Euclidean at infinity. Then there exist a closed abelian subgroup $K$ of $O(k)$ and an integer $l\in[0,k]$ such that for any equivariant tangent cone of $(\widetilde{M},\pi_1(M,x))$ at infinity $(\widetilde{Y},\tilde{y},G)=(\mathbb{R}^k\times C(Z),(0,z),G)$, the projected $G$-action on $\mathbb{R}^k$-factor $(\mathbb{R}^k,0,p(G))$ satisfies that $p(G)=K\times \mathbb{R}^l$, with $K$ fixing $0$ and the subgroup $\{e\}\times \mathbb{R}^l$ acting as translations in $\mathbb{R}^k$.
\end{thm4}

\begin{rem}
	Theorem B can be generalized to nilpotent fundamental groups; see Remark \ref{rem_non_abelian}. We do not pursue this general statement, since we do not need this in the present paper.
\end{rem}

Theorem A follows from Theorem B. Suppose that $\pi_1(M,x)$ is not finitely generated, then without lose of generality, we can assume that $\pi_1(M)$ is abelian \cite{Wi}. Using the lengths of Gromov's short generators $r_i\to\infty$ \cite{Gro1}, we get an equivariant tangent cone of $(\widetilde{M},\pi_1(M,x))$ at infinity $(\widetilde{Y},\tilde{y},G)$, where $G$-orbit at $\tilde{y}$ is not connected (see Lemma \ref{non_cnt_orb}). If $\widetilde{M}$ is $k$-Euclidean at infinity, then by Theorem B the orbit $G\cdot\tilde{y}$ is connected, a contradiction.

We illustrate our approach to Theorem B. Put $\Gamma=\pi_1(M,x)$. Given two equivariant tangent cones of $(\widetilde{M},\Gamma)$ at infinity $(\widetilde{Y}_i,\tilde{y}_i,G_i)$ $(i=1,2)$, assume that their projected actions $(\mathbb{R}^k,0,p(G_i))$ are different. We consider the set of all equivariant tangent cones of $(\widetilde{M},\Gamma)$ at infinity $\Omega(\widetilde{M},\Gamma)$. It is known that $\Omega(\widetilde{M},\Gamma)$ is compact and connected in the equivariant Gromov-Hausdorff topology. Consequently, for any $\epsilon>0$, there are finitely many spaces $(W_j,w_j,H_j)\in\Omega(\widetilde{M},\Gamma)$ ($j=1,...,l$) such that $(W_1,w_1,H_1)=(\widetilde{Y}_1,\tilde{y}_1,G_1)$, $(W_l,w_l,H_l)=(\widetilde{Y}_2,\tilde{y}_2,G_2)$, and
$$d_{GH}((W_j,w_j,H_j),(W_{j+1},w_{j+1},H_{j+1}))\le \epsilon$$
for all $j=1,...,l-1$. When $\widetilde{M}$ is $k$-Euclidean at infinity, $W_i=\mathbb{R}^k\times C(Z_i)$ with $\mathrm{diam}(Z_i)\le \pi-\eta(\widetilde{M})$. Under this control, the associated chain of $\mathbb{R}^k$-factor in $W_j$ with projected $p(H_j)$-action $\{(\mathbb{R}^k,0,p(H_j))\}_{j=1}^l$ form a $\psi(\epsilon)$ chain, where $\psi(\epsilon)$ is a positive function with $\psi(\epsilon)\to 0$ as $\epsilon\to 0$.

To see a contradiction without involving the complexity in general situation, we restrict to
the special case that all $p(H_j)$-actions fix $0$. Then this leads to the following stability of isometric actions on the unit sphere $S^{k-1}\subseteq \mathbb{R}^k$: if $(S^{k-1},K_1)$ and $(S^{k-1},K_2)$ are sufficiently close in the equivariant Gromov-Hausdorff topology, then either $K_1$ and $K_2$ are conjugate in $O(k)$, or $\dim(K_1)\not=\dim(K_2)$ (see Proposition \ref{stability}). It turns out this stability is enough for us to derive a contradiction. For instance, if $p(G_1)=\{e\}$ and $p(G_2)=\mathbb{Z}_2$, there is no $\epsilon$-chain $\{(\mathbb{R}^k,0,p(H_j))\}_{j=1}^l$, with $p(H_i)$ fixing $0$, between $(\mathbb{R}^k,0,\{e\})$ and $(\mathbb{R}^k,0,\mathbb{Z}_2)$, given that $\epsilon$ is small (see Lemma \ref{gap_type_0}). To deal with the general situation where these $p(H_i)$-action may not fix $0$, we develop a key technical tool, referred as critical rescaling (see Section 2 for details).

We start with some preliminaries in Section 1. To illustrate the critical rescaling argument used in the proof of Theorem B, we consider a rudimentary version of Theorem B in section 2 (see Proposition \ref{not_type_0}). In Section 3, we prove a stability result on the isometric actions, which is another main ingredient in the proof of Theorem B as mentioned above. Afterwards, We prove Theorem B in Section 4.

The author would like to thank Professor Xiaochun Rong for his assistance during the preparation of this paper. Part of the paper was written during the author's visit in Capital Normal University at Beijing. The author would like to thank Capital Normal University for support and hospitality.

\tableofcontents

\section{Preliminaries}

In this section, we supply notions and results that will be used through the rest of this paper.

\noindent\textbf{a. Tangent cones at infinity. }Let $(M,x)$ be an open $n$-manifold of $\mathrm{Ric}\ge 0$, and let $(\widetilde{M},\tilde{x})$ be its Riemannian universal cover.  By path lifting $\Gamma=\pi_1(M,x)$ acts on $\widetilde{M}$ isometrically, freely and discretely. For any sequence $r_i\to\infty$, we can pass to a subsequence and obtain equivariant Gromov-Hausdorff convergence \cite{FY}:
\begin{center}
	$\begin{CD}
	(r^{-1}_i\widetilde{M},\tilde{x},\Gamma) @>GH>>
	(\widetilde{Y},\tilde{y},G)\\
	@VV\pi V @VV\pi V\\
	(r^{-1}_iM,x) @>GH>> (Y=\widetilde{Y}/G,y),
	\end{CD}$
\end{center}
where $G$ is a closed subgroup of $\mathrm{Isom}(\widetilde{Y})$, the isometry group of $\widetilde{Y}$.  We call the above space $(\widetilde{Y},\tilde{y},G)$ an equivariant tangent cone of $(\widetilde{M},\Gamma)$ at infinity. $\mathrm{Isom}(\widetilde{Y})$ is a Lie group \cite{CC3,CN1}, thus $G$ is a Lie group. For convenience, we introduce the set of all tangent cones of ${M}$ at infinity:
	$$\Omega({M})=\{({Y},{y})\ |\ ({Y},{y}) \text{ is a tangent cone of } {M} \text{ at infinity}\}.$$
The set $\Omega({M})$ has a natural topology: Gromov-Hausdorff topology. Similarly, we can consider the set of all equivariant tangent cones of $(\widetilde{M},\Gamma)$ at infinity:
	$$\Omega(\widetilde{M},\Gamma)=\{(\widetilde{Y},\tilde{y},G)|(\widetilde{Y},\tilde{y},G) \text{ is an equivariant tangent cone of } (\widetilde{M},\Gamma) \text{ at infinity}\}$$
endowed with the equivariant Gromov-Hausdorff topology. We use the following fact implicitly in the proof of Theorem B.
\begin{prop}
	Let $(M,x)$ be an open $n$-manifold with $\mathrm{Ric}\ge 0$. Then the set $\Omega(\widetilde{M},\Gamma)$ is compact and connected in the equivariant Gromov-Hausdorff topology.
\end{prop}

Throughout this paper, we use $d_{GH}$ to denote Gromov-Hausdorff distance, or equivariant Gromov-Hausdorff distance, depending on the context.

\noindent\textbf{b. Short generators }To study whether the fundamental group is finitely generated, we use Gromov's short generators of $\Gamma=\pi_1(M,x)$ \cite{Gro1}. We say that $\{\gamma_1,...,\gamma_i,...\}$ is a set of short generators of $\Gamma$, if
\begin{center}
	$d(\gamma_1\tilde{x},\tilde{x})\le d(\gamma\tilde{x},\tilde{x})$ for all $\gamma\in\Gamma$,
\end{center}
and for each $i\ge 2$,
\begin{center}
	$d(\gamma_i\tilde{x},\tilde{x})\le d(\gamma\tilde{x},\tilde{x})$ for all $\gamma\in\Gamma-\langle\gamma_1,...,\gamma_{i-1}\rangle$,
\end{center}
where $\langle\gamma_1,...,\gamma_{i-1}\rangle$ is the subgroup generated by $\gamma_1,...,\gamma_{i-1}$.

If $\Gamma$ has infinitely many short generators, then we can use the lengths of short generators to get an equivariant tangent cone of $(\widetilde{M},\Gamma)$ at infinity. In this way, we find a special element in $\Omega(\widetilde{M},\Gamma)$, whose orbit at $\tilde{y}$ is not connected.

\begin{lem}\label{non_cnt_orb}
	Let $(M,x)$ be an open $n$-manifold with $\mathrm{Ric}\ge 0$. Suppose that $\Gamma$ has infinitely many short generators $\{\gamma_1,...,\gamma_i,...\}$. Then in the following equivariant tangent cone of $(\widetilde{M},\tilde{x},\Gamma)$ at infinity
	$$(r_i^{-1}\widetilde{M},\tilde{x},\Gamma)\overset{GH}\longrightarrow(\widetilde{Y},\tilde{y},G),$$
	the orbit $G\cdot\tilde{y}$ is not connected, where $r_i=d(\gamma_i\tilde{x},\tilde{x})\to\infty$.
\end{lem}

Lemma \ref{non_cnt_orb} follows directly from the definition of short generators and equivariant Gromov-Hausdorff convergence (see \cite{Pan} for the proof of Lemma \ref{non_cnt_orb}). 

Another theorem related to the Milnor conjecture is the Wilking's reduction.

\begin{thm}\cite{Wi}\label{red_W}
	Let $M$ be an open manifold with $\mathrm{Ric}\ge 0$.
	If $\pi_1(M)$ is not finitely generated, then it contains a non-finitely generated abelian subgroup.
\end{thm}

\noindent\textbf{c. Structures of Ricci limit spaces }Next we recall some results on Ricci limit spaces. We denote $\mathcal{M}(n,0)$ as the set of all the Gromov-Hausdorff limit spaces coming from some sequence $(M_i,x_i)$ of $n$-manifolds with $\mathrm{Ric}\ge 0$.

\begin{thm}\cite{CC1}\label{splitting}
	Let $(X,x)\in\mathcal{M}(n,0)$ be a Ricci limit space. If $X$ contains a line, then $X$ splits isometrically as $\mathbb{R}\times Y$.
\end{thm}

\begin{thm}\cite{CC1}\label{vol_cone}
	Let $(M,x)$ be an open $n$-manifold of $\mathrm{Ric}\ge0$. If $M$ has Euclidean volume growth, then any tangent cone of $M$ at infinity $(Y,y)$ is a metric cone $C(Z)$ with vertex $y$ of Hausdorff dimension $n$. Moreover, $\mathrm{diam}(Z)\le \pi$.
\end{thm}

When $(Y,y)\in \Omega(M)$ is a metric cone $Y=\mathbb{R}^k\times C(Z)$, we always put $0\in\mathbb{R}^k$ as the projection of $y\in {Y}$ to the Euclidean factor.

The metric cone structure and Theorem \ref{splitting} immediately imply the following properties on its isometries.

\begin{prop}\label{cone_spit}
	Let $(\mathbb{R}^k\times C(Z),(0,z))\in\mathcal{M}(n,0)$ be a metric cone, where $C(Z)$ has vertex $z$ and $\mathrm{diam}(Z)<\pi$. Then for any isometry $g$ of $C(Z)$, we have $g\cdot (\mathbb{R}^k\times \{z\})\subseteq \mathbb{R}^k\times \{z\}$.
\end{prop}

\begin{prop}\label{isom_gp_split}
	Let $Y=\mathbb{R}^k\times C(Z)\in \mathcal{M}(n,0)$ be a metric cone with $\mathrm{diam}(Z)<\pi$. Then its isometry group $\mathrm{Isom}(Y)$ splits as $\mathrm{Isom}(\mathbb{R}^k)\times \mathrm{Isom}(Z)$.
\end{prop}

Due to this proposition, for a metric cone $Y=\mathbb{R}^k\times C(Z)\in \mathcal{M}(n,0)$, where $\mathrm{diam}(Z)<\pi$, there is a natural projection map:
$$p:\mathrm{Isom}(Y)\to\mathrm{Isom}(\mathbb{R}^k).$$
Throughout this paper, we always use $p$ to denote this projection map.

\begin{rem}\label{rem_orb}
	For any $v\in\mathbb{R}^k$, the orbit $G\cdot (v,z)=(p(G)\cdot v,z)$ can be naturally identified as $p(G)\cdot v \subseteq \mathbb{R}^k \times \{z\}$.
\end{rem}

Theorem A directly follows from these preparations and equivariant stability Theorem B.

\begin{proof}[Proof of Theorem A by assuming Theorem B]
	Suppose that $\pi_1(M)$ is not finitely generated, then by Lemma \ref{non_cnt_orb}, there is $(\mathbb{R}^k\times C(Z),(0,z),G)\in \Omega(\widetilde{M},\Gamma)$ such that the orbit $G\cdot (0,z)=(p(G)\cdot 0,z)$ is not connected. This contradicts Theorem B.
\end{proof}

When $\mathrm{diam}(Z)<\pi$, $C(Z)$ has no lines, and Proposition \ref{cone_spit} says that any isometry of $C(Z)$ must fix its vertex. This simple observation implies that if $\widetilde{M}$ is $0$-Euclidean at infinity, then its fundamental group must be finite.

\begin{prop}\label{0_sym}
	Let $M$ be an open $n$-manifold of $\mathrm{Ric}\ge 0$. If $\widetilde{M}$ is $0$-Euclidean at infinity, then $\pi_1(M)$ is finite.
\end{prop}

\begin{proof}
	Suppose that $\Gamma=\pi_1(M,x)$ is an infinite group, then there are elements $\gamma_i\in \pi_1(M,x)$ with $r_i:=d(\gamma_i\tilde{x},\tilde{x})\to\infty$. Passing to a subsequence if necessary, we consider an equivariant tangent cone of $(\widetilde{M},\Gamma)$ at infinity:
	$$(r_i^{-1}\widetilde{M},\tilde{x},\Gamma)\overset{GH}\longrightarrow(\widetilde{Y},\tilde{y},G).$$
	By our choice of $r_i$, there is $g\in G$ such that $d(g\cdot\tilde{y},\tilde{y})=1$. On the other hand, since $\widetilde{M}$ is $0$-Euclidean at infinity, $(\widetilde{Y},\tilde{y})$ is a metric cone with no lines and $\tilde{y}$ is the unique vertex. Thus the orbit $G\cdot \tilde{y}$ must be a single point $\tilde{y}$ by Proposition \ref{cone_spit}. A contradiction.
\end{proof}

\begin{cor}\label{non_max_diam_growth}
	Let $(M,x)$ be an open $n$-manifold of $\mathrm{Ric}\ge 0$. If its universal cover $(\widetilde{M},\tilde{x})$ has Euclidean volume growth and non-maximal diameter growth
	$$\limsup\limits_{R\to\infty} \dfrac{\mathrm{diam}(\partial B_R(\tilde{x}))}{R}<2,$$
	then $\pi_1(M,x)$ is a finite group. Consequently, $M$ itself has Euclidean volume growth.
\end{cor}

Here we use extrinsic metric on $\mathrm{diam}(\partial B_R(\tilde{x}))$, so we always have
$$\dfrac{\mathrm{diam}(\partial B_R(\tilde{x}))}{R}\le 2.$$

\section{A Critical Rescaling Argument}

In this section, we develop the critical rescaling argument, a key technical tool as mentioned in the introduction, to prove a special case of Theorem B: if there is $(\widetilde{Y},\tilde{y},G)\in\Omega(\widetilde{M},\Gamma)$ such that $p(G)$ is trivial, then for any $(\widetilde{W},\tilde{w},H)\in\Omega(\widetilde{M},\Gamma)$, $p(H)$ is also trivial (see Proposition \ref{not_type_0}). The proof of Theorem B is also modeled on the proofs in this section.

We first show that if $M$ is $k$-Euclidean at infinity, then for any $Y=\mathbb{R}^k\times C(Z)\in \Omega(M)$, there is a uniform gap between $Y$ and any Ricci limit space splitting off a $\mathbb{R}^{k+1}$-factor. This is indeed a direct consequence of being $k$-Euclidean at infinity.

\begin{lem}\label{gap_higher_sym}
	Let $M$ be an open $n$-manifold of $\mathrm{Ric}\ge 0$. If $M$ is $k$-Euclidean at infinity, then there is $\epsilon(M)>0$ such that for any $(Y,y)\in\Omega(M)$ and any Ricci limit space $\mathbb{R}^l\times X\in\mathcal{M}(n,0)$ with $l>k$, we have
	$$d_{GH}((Y,y),(\mathbb{R}^l\times X,(0,x)))\ge \epsilon(M).$$
\end{lem}

\begin{proof}
	Suppose the contrary, then we would have a sequence $(Y_i,y_i)\in\Omega(M)$ such that as $i\to\infty$,
	$$d_{GH}((Y_i,y_i),(\mathbb{R}^{l_i}\times X_i),(0,x_i))\to 0,$$
	where $\mathbb{R}^{l_i}\times X_i\in \mathcal{M}(n,0)$ and $k<l_i\le n$. By pre-compactness, we can pass to a subsequence and have convergence
	$$(\mathbb{R}^{l_i}\times X_i,(0,x_i))\overset{GH}\longrightarrow(\mathbb{R}^{l_\infty}\times X_\infty,(0,x_\infty))$$
	with integer $l_\infty>k$. For the corresponding subsequence of $(Y_i,y_i)$, it has the same limit. By a standard diagonal argument, $(\mathbb{R}^{l_\infty}\times X_\infty,(0,x_\infty))$ is also a tangent cone of $M$ at infinity. This is a contradiction to the assumption that $M$ is $k$-Euclidean at infinity.
\end{proof}

\begin{rem}\label{rem_gap_diam}
	Note that for a metric cone $C(Z)\in\mathcal{M}(n,0)$, $C(Z)$ splits off a line if and only if $\mathrm{diam}(Z)=\pi$. From this perspective, Lemma \ref{gap_higher_sym} implies that if $M$ is $k$-Euclidean at infinity, then there exists $\eta(M)>0$ such that for any $\mathbb{R}^k\times C(Z)\in \Omega(M)$, $Z$ has diameter no more than $\pi-\eta(M)$.
\end{rem}

Next we prove a gap phenomenon between two classes of group actions on spaces in $\Omega(M)$, which is a key property needed in the critical rescaling argument.

\begin{lem}\label{gap_type_0}
	Let $M$ be an open $n$-manifold of $\mathrm{Ric}\ge 0$. Suppose that $\widetilde{M}$ is $k$-Euclidean at infinity. Then there exists a constant $\epsilon(M)>0$ such that the following holds.
	
	For two spaces $(\widetilde{Y}_j,\tilde{y}_j,G_j)\in \Omega(\widetilde{M},\Gamma)$ with $(\widetilde{Y}_j,\tilde{y}_j)=(\mathbb{R}^k\times C(Z_j),(0,z_j))$ $(j=1,2)$ if\\
	(1) $p(G_1)$ is trivial, and\\
	(2) there is $g\in G_2$ such that $p(g)\not= e$ and $d(g\cdot \tilde{y}_2,\tilde{y}_2)\le 1$,\\
	then
	$$d_{GH}((\widetilde{Y}_1,\tilde{y}_1,G_1),(\widetilde{Y}_2,\tilde{y}_2,G_2))\ge \epsilon(M).$$
\end{lem}

For the Euclidean space $\mathbb{R}^k$ with isometric $G$-action and a non-identity element $g\in G$, if $d(g\cdot 0,0)\le 1$, then it is obvious that
$$d_{GH}((\mathbb{R}^k,0,G),(\mathbb{R}^k,0,\{e\}))\ge 1/2.$$
Let $(\widetilde{Y}_1,\tilde{y}_1,G_1)$ and $(\widetilde{Y}_2,\tilde{y}_2,G_2)$ be two spaces in $\Omega(\widetilde{M},\Gamma)$ as in Lemma \ref{gap_type_0}. Roughly speaking Lemma \ref{gap_higher_sym} assures that for $\epsilon$ sufficiently small, any $\epsilon$-approximation from $(\widetilde{Y}_1,\tilde{y}_1)$ to $(\widetilde{Y}_2,\tilde{y}_2)$ can not map $\mathbb{R}^k$-factor to non-Euclidean cone factor $C(Z_2)$. In other words, an $\epsilon$-approximation map should map $\mathbb{R}^k$-factor to $\mathbb{R}^k$-factor. Together with the $p(G_j)$-action on $\mathbb{R}^k$-factor, we see that there should be a gap between $(\widetilde{Y}_1,\tilde{y}_1,G_1)$ and $(\widetilde{Y}_2,\tilde{y}_2,G_2)$.

\begin{proof}[Proof of Lemma \ref{gap_type_0}]
	Suppose the contrary, then we have two sequences in $\Omega(M,\Gamma)$: $\{(\widetilde{Y}_{i1},\tilde{y}_{i1},G_{i1})\}$ and $\{(\widetilde{Y}_{i2},\tilde{y}_{i2},G_{i2})\}$ such that\\
	(1) $p(G_{i1})$ is trivial,\\
	(2) there is $g_i\in G_{i2}$ such that $p(g_i)\not=e$ and $d(g_i\cdot \tilde{y}_{i2},\tilde{y}_{i2})\le 1$,\\
	(3) $d_{GH}((\widetilde{Y}_{i1},\tilde{y}_{i1},G_{i1}),(\widetilde{Y}_{i2},\tilde{y}_{i2},G_{i2}))\to 0$ as $i\to\infty$.\\
	Passing to some subsequences if necessary, the above two sequences converge to the same limit:
	$$(\widetilde{Y}_{i1},\tilde{y}_{i1},G_{i1})\overset{GH}\longrightarrow(\widetilde{Y}_\infty,\tilde{y}_\infty,G_\infty),$$
	$$(\widetilde{Y}_{i2},\tilde{y}_{i2},G_{i2})\overset{GH}\longrightarrow(\widetilde{Y}_\infty,\tilde{y}_\infty,G_\infty),$$
	with $(\widetilde{Y}_\infty,\tilde{y}_\infty)=(\mathbb{R}^k\times C(Z_\infty),(0,z_\infty))$. By Lemma \ref{gap_higher_sym}, $C(Z_\infty)$ does not split off any line, and thus
	$$(\mathbb{R}^k\times \{z_{i1}\},\tilde{y}_{i1},p(G_{i1}))\overset{GH}\longrightarrow(\mathbb{R}^k\times \{z_\infty\},\tilde{y}_\infty,p(G_\infty)),$$
	$$(\mathbb{R}^k\times \{z_{i2}\},\tilde{y}_{i2},p(G_{i2}))\overset{GH}\longrightarrow(\mathbb{R}^k\times \{z_\infty\},\tilde{y}_\infty,p(G_\infty))$$
	(cf. Remark \ref{rem_orb}). From the first sequence, we see that $p(G_\infty)$ is trivial because $p(G_{i1})=\{e\}$. On the other hand, $p(G_{i2})$ contains some element $\beta_i$ with $d_{\mathbb{R}^k}(\beta_i\cdot0,0)\le 1$. $\beta_i$ sub-converges to some element $\beta_\infty \in G_\infty$ with $d(\beta_\infty\cdot 0,0)\le 1$. If $\beta_\infty\not =e$, then $p(G_\infty)$ is non-trivial. If $\beta_\infty=e$, then we consider the subgroup $H_i=\langle\beta_i\rangle$. The sequence of subgroups $H_i$ sub-converges to some non-trivial subgroup $H_\infty$ of $p(G_\infty)$ because $D_1(H_i)\ge 1/20$, where $D_1(H_i)$ is the displacement of $H_i$ on $B_1(0)\subseteq \mathbb{R}^k$. In either case, $p(G_\infty)$, a contradiction.
\end{proof}

\begin{rem}
	The gap in Lemma \ref{gap_higher_sym} plays a key role in the above proof; it guarantees that symmetries on the non-Euclidean cone factor and on the Euclidean factor can not interchange. If there is no gap between the non-Euclidean cone factor $C(Z)$ and spaces splitting off lines, then Lemma \ref{gap_type_0} fails.  As an example, we construct a continuous family of metric cones $(Y_t,y_t,G_t)$ $(-\delta\le t\le\delta)$ such that $Y_t=\mathbb{R}^2\times C(Z_t)$, where $\mathrm{diam}(Z_t)\le \pi$. As $t\to 0$,
	$$d_{GH}((\mathbb{R}^2\times \{z_{-t}\},(0,z_{-t}),p(G_{-t})),(\mathbb{R}^2\times z_t,(0,z_t),p(G_t)))\not\to 0.$$
	For $|t|<\delta$ small, we put $Y_t=\mathbb{R}^2\times C(S^1_t)$, where $S^1_t$ is the round circle of diameter $\pi-|t|$. When $t=0$, then $Y_t=\mathbb{R}^4$. Next we define $G_t$-action on $Y_t$. For $t>0$, $G_t=S^1$ acting as rotations on the $C(S^1_t)$-factor; while for $t\le 0$, $G_t=S^1$ acting as rotations about the origin on $\mathbb{R}^2$-factor. It is clear that $(Y_t,y_t,G_t)$ is a continuous path in the equivariant Gromov-Hausdorff topology. However, $p(G_t)$ is trivial for $t>0$ while $p(G_t)=S^1$ for $t<0$; they can not be arbitrarily close as $t\to 0$.
\end{rem}

We are ready to prove the following rudimentary version of Theorem B.

\begin{prop}\label{not_type_0}
	Let $(M,x)$ be an open $n$-manifold of $\mathrm{Ric}\ge 0$, whose universal cover is $k$-Euclidean at infinity. If there is $(\widetilde{Y},\tilde{y},G)\in\Omega(\widetilde{M},\Gamma)$ such that $p(G)$ is trivial, then for any space $(\widetilde{W},\tilde{w},H)\in\Omega(\widetilde{M},\Gamma)$, $p(H)$ is also trivial.
\end{prop}

\begin{proof}
	We argue by contradiction. Suppose that there are $r_i\to\infty$ and $s_i\to\infty$ such that
	$$(r_{i}^{-1}\widetilde{M},\tilde{x},\Gamma)\overset{GH}\longrightarrow(\widetilde{Y}_1,\tilde{y}_1,G_1),$$
	$$(s_{i}^{-1}\widetilde{M},\tilde{x},\Gamma)\overset{GH}\longrightarrow(\widetilde{Y}_2,\tilde{y}_2,G_2),$$
	where $p(G_1)$ is trivial but $p(G_2)$ is not. Scaling down the sequence $s_i^{-1}$ by a constant if necessary, we assume that there is $g\in G_2$ such that $p(g_2)\not=e$ and $d(g\cdot\tilde{y}_2,\tilde{y}_2)\le 1$. We pass to a subsequence and assume that $t_i:=s_i^{-1}/r_i^{-1}\to \infty$. This enables us to regard the above first sequence as a rescaling of the second one. Put
	$$(N_i,q_i,\Gamma_i)=(s_i^{-1}\widetilde{M},\tilde{x},\Gamma).$$
	In this way, we can rewrite these two convergent sequences as ($t_i\to\infty$):
	$$(N_i,q_i,\Gamma_i)\overset{GH}\longrightarrow (\widetilde{Y}_1,\tilde{y}_1,G_1),$$
	$$(t_iN_i,q_i,\Gamma_i)\overset{GH}\longrightarrow (\widetilde{Y}_2,\tilde{y}_2,G_2).$$
	
	We look for a contradiction in some intermediate rescaling sequence. For each $i$, we define a set of scales
	\begin{align*}
		L_i:=\{\  1\le l\le t_i\ |\ & d_{GH}((l{N}_i,q_i,\Gamma_i),(W,w,H))\le \epsilon/3 \\
		& \text{ for some space $(W,w,H)\in \Omega(\widetilde{M},\Gamma)$}\\
		& \text{ such that $H$ has some element $h$}\\
		& \text{ with } p(h)\not= e \text{ and } d(h\cdot \tilde{w},\tilde{w})\le 1 \},
	\end{align*}
	where $\epsilon=\epsilon({M})>0$ is the constant in Lemma \ref{gap_type_0}. It is clear that $t_i\in L_i$ for all $i$ large, thus $L_i$ is non-empty. We choose $l_i\in L_i$ such that $\inf L_i\le l_i\le \inf L_i+1/i$. We regard this $l_i$ as the critical rescaling sequence.
	
	\textbf{Claim 1:} $l_i\to\infty$. Suppose that $l_i$ subconverges to $C<\infty$, then for this subsequence, we can pass to a subsequence again and obtain the convergence
	$$(l_iN_i,q_i,\Gamma_i)\overset{GH}\longrightarrow (C\cdot\widetilde{Y}_1,\tilde{y}_1,G_1).$$
	Since $l_i\in L_i$, by definition of $L_i$ and the above convergence, we conclude that
	$$d_{GH}((C\cdot\widetilde{Y}_1,\tilde{y}_1,G_1),(W,w,H))\le \epsilon/2$$
	for some space $(W,w,H)$ such that there is $h\in H$ with $p(h)\not= e$ and $d(h\cdot\tilde{w},\tilde{w})\le 1$. On the other hand, on $(C\cdot\widetilde{Y}_1,\tilde{y}_1,G_1)$, $p(G_1)$ is trivial. This is a contradiction to the choice of $\epsilon$ and Lemma \ref{gap_type_0}. Hence Claim 1 is true.
	
	Next we consider the convergence
	$$(l_iN_i,q_i,\Gamma_i)\overset{GH}\longrightarrow (\widetilde{Y}',\tilde{y}',G')\in\Omega(\widetilde{M},\Gamma).$$
	We will derive a contradiction by ruling out all the possibilities of $p(G')$-action on the $\mathbb{R}^k$-factor of $\widetilde{Y}'$.
	
	\textbf{Claim 2:} $p(G')$ is non-trivial. For each $i$, because $l_i\in L_i$, we know that
	$$d_{GH}((l_iN_i,q_i,\Gamma_i),({W}_i,{w}_i,K_i))\le \epsilon/3$$
	for some $({W}_i,{w}_i)\in \Omega(\widetilde{M})$ with $K_i$-action such that there is $k_i\in K_i$ with $p(k_i)\not= e$ and $d(k_i\cdot\tilde{w}_i,\tilde{w}_i)\le 1$. Since $(l_iN_i,q_i,\Gamma_i)$ converges to $(\widetilde{Y}',\tilde{y}',G')$, the limit space satisfies
	$$d_{GH}((\widetilde{Y}',\tilde{y}',G'),({W}_i,{w}_i,K_i))\le \epsilon/2$$
	for $i$ large. By Lemma \ref{gap_type_0}, $p(G')$ is nontrivial.
	
	By Claim 2, there is some $g'\in G'$ such that $p(g')\not= e$. We put $d:=d(g'\cdot \tilde{y}',\tilde{y}')$. If $d\le 1$, we consider the scaling sequence $l_i/2$:
	$$(l_i/2 \cdot N_i,q_i,\Gamma_i)\overset{GH}\longrightarrow(1/2 \cdot\widetilde{Y}',\tilde{y}',G').$$
	Note that on $(1/2 \cdot\widetilde{Y}',\tilde{y}',G')$, there is some element $g'\in G'$ with $p(g')\not=e$ and $d(g'\cdot\tilde{y}',\tilde{y}')\le 1/2$.
	This shows that $l_i/2\in L_i$ for $i$ large, which is a contradiction our choice of $l_i$ with $\inf L_i\le l_i\le \inf L_i+1/i$. If $d>1$, then we consider the scaling sequence $l_i/(2d)$:
	$$(l_i/(2d) \cdot N_i,q_i,\Gamma_i)\overset{GH}\longrightarrow(1/(2d) \cdot\widetilde{Y},\tilde{y},H').$$
	and a similar contradiction would arise because $l_i/(2d)\in L_i$. In any case, we see a contradiction. This completes the proof.
\end{proof}

\begin{rem}\label{rem_scales_2}
  In the above proof when defining $L_i$, we include all the contradictory $H$-actions with $p(H)\not= \{e\}$ and a constraint on its displacement. In particular, this allows $p(H)$ to be different from $p(G_2)$. Doing so is necessary. For example, if $p(G_2)=\mathbb{Z}$ and we require $p(H)=\mathbb{Z}$ when defining $L_i$, then the intermediate sequence
	$$(l_iN_i,q_i,\Gamma_i)\overset{GH}\longrightarrow (\widetilde{Y}',\tilde{y}',G')$$
	may have a limit group with $p(G')=\mathbb{Z}_p$, where $p$ is a large integer. This $\mathbb{Z}_p$ acts on $\mathbb{R}^k$ as rotations on a plane about a point far away from $0$, so that $(\mathbb{R}^k,0,\mathbb{Z}_p)$ and $(\mathbb{R}^k,0,\mathbb{Z})$ are very close. In this situation, one can not derive a contradiction by dividing $l_i$ by a constant.
\end{rem}

To conclude this section, we use Proposition \ref{not_type_0} to prove Corollary \ref{n_unique_up} in dimension $3$.

\begin{cor}\label{3_unique}
	Let $M$ be an open $3$-manifold with $\mathrm{Ric}\ge 0$. Suppose that $\widetilde{M}$ has Euclidean volume growth and the unique tangent cone at infinity. Then $\pi_1(M)$ is finitely generated.
\end{cor}

\begin{proof}
	We argue by contradiction. Suppose that $\Gamma$ has infinitely many short generators $\{\gamma_1,...,\gamma_i,...\}$. Let $r_i=d(\gamma_i\cdot\tilde{x},\tilde{x})\to\infty$. Passing to a subsequence, we consider an equivariant tangent cone of $(\widetilde{M},\Gamma)$ at infinity:
	$$(r_{i(j)}^{-1}\widetilde{M},\tilde{x},\Gamma)\overset{GH}\longrightarrow(\widetilde{Y},\tilde{y},G).$$
	We know that $\widetilde{Y}$ is a metric cone of Hausdorff dimension $3$ (Theorem \ref{vol_cone}), and orbit $G \cdot\tilde{y}$ is not connected (Lemma \ref{non_cnt_orb}). By Theorem \ref{red_W}, we assume that $\Gamma$ is abelian.
	
	$\widetilde{Y}$ is isometric to $\mathbb{R}^k\times C(Z)$ with vertex $\tilde{y}=(0,z)$, where $k\in\{0,1,2,3\}$, $C(Z)$ has vertex $z$ and $\mathrm{diam}(Z)<\pi$. If $k=0$, then $\Gamma$ is finite (Proposition \ref{0_sym}). If $k=3$, then $\widetilde{Y}=\mathbb{R}^3$, and consequently $\widetilde{M}$ is isometric to $\mathbb{R}^3$ \cite{Co}. If $k=2$, then $\widetilde{Y}$ is actually isometric to $\mathbb{R}^3$ according to the fact that the singular set of $\widetilde{Y}$ has co-dimension at least $2$ \cite{CC2}. It remains to handle the case $\widetilde{Y}=\mathbb{R}\times C(Z)$. Since the non-connected orbit $G \cdot\tilde{y}$ is contained in $\mathbb{R}\times \{z\}$, together with the assumption that $G$ is abelian, we see that the orbit $G\cdot \tilde{y}$ is either a $\mathbb{Z}$ translation orbit, or a $\mathbb{Z}_2$ reflection orbit in $\mathbb{R}\times \{z\}$. In either case, $p(G)$-action is free at $0$. For the equivariant tangent cone of $(\widetilde{Y},\tilde{y},G)$ at $\tilde{y}$ ($j\to\infty$):
	$$(j\widetilde{Y},\tilde{y},G)\overset{GH}\longrightarrow(\widetilde{Y},\tilde{y},H),$$
	it is clear that $p(H)=\{e\}$. By a standard diagonal argument, we can find $s_j\to \infty$ such that
	$$(s_j^{-1}\widetilde{M},\tilde{x},\Gamma)\overset{GH}\longrightarrow(\widetilde{Y},\tilde{y},H).$$
	Right now there are two equivariant tangent cones of $(\widetilde{M},\Gamma)$ at infinity: $(\widetilde{Y},\tilde{y},H)$ and $(\widetilde{Y},\tilde{y},G)$,
	with $p(H)=\{e\}$ but $p(G)\not=\{e\}$, a contradiction to Proposition \ref{not_type_0}.
\end{proof}

\section{Stability of Isometric Actions}

In this section, for an isometric $G$-action on a Riemannian manifold $M$, we always assume that $G$ is a closed subgroup of $\mathrm{Isom}(M)$. The goal is the following stability result on isometric actions on any compact manifold $M$, which will be used in the proof of Theorem B with $M$ being the unit sphere $S^{k-1}$.

\begin{prop}\label{stability}
	Let $(M,G)$ be a compact Riemannian manifold with isometric $G$-action. Then there exists a constant $\epsilon>0$, depending on $(M,G)$, such that the following holds.
	
	For any isometric $H$-action on $M$, if
	$$d_{GH}((M,G),(M,H))\le\epsilon,$$
	then either $H$-action is conjugate to $G$-action by an isometry, or $\dim(H)<\dim(G)$.
\end{prop}

One may compare Proposition \ref{stability} with the result below by Grove and Karcher \cite{GK}.

\begin{thm}\label{stable_GK}
	Let $M$ a compact Riemannian manifold. Then there exists $\epsilon(M)>0$ such that for any two isometric $G$-actions
	$$\mu_1,\mu_2: G\times M\to M$$
	with $d_M(\mu_1(g,x),\mu_2(g,x))\le\epsilon(M)$ for all $g\in G$ and $x\in M$, these two actions are conjugate by an isometry.
\end{thm}

We mention that the stability of group actions can be traced back to Palais \cite{Pal}. He shows that any two $C^1$-close $G$-actions, as diffeomorphisms on $M$, can be conjugated by a diffeomorphism, where $G$ is a compact Lie group. Grove and Karcher use the center of mass technique, and explicitly construct the conjugation map. They also interpret the $C^1$-closeness in terms of curvature bounds of $M$, 
when one of actions is by isometries. For our purpose, we restrict our attention to isometric actions only here.

Proposition \ref{stability} is different from Theorem \ref{stable_GK} in the following aspects. Proposition \ref{stability} considers two isometric actions with possibly different groups. For instance, $G=S^1$ and we can take $H=\mathbb{Z}_p\subseteq G$ with large integer $p$. Even if one assume $G=H$, the closeness of these two actions in the equivariant Gromov-Hausdorff topology is weaker than the pointwise closeness condition in Theorem \ref{stable_GK}. For example, we know that there is a sequence of circle actions on the standard torus $T^2=S^1\times S^1$ converging to $T^2$-action:
$$(T^2,S_i^1)\overset{GH}\longrightarrow(T^2,T^2).$$
Thus for any $\epsilon>0$, we can find two different circle actions in the tail of this sequence such that
$$d_{GH}((T^2,S^{1}_{j}),(T^2,S^1_{k}))\le\epsilon,$$
where $j,k$ are sufficiently large.
However, these circle actions are not pointwise close. This example also illustrates that the $\epsilon$ in Proposition \ref{stability} has to depend on the $G$-action.

To prove Proposition \ref{stability}, we recall some facts on equivariant Gromov-Hausdorff convergence \cite{FY}. Given $(M,H_i)\overset{GH}\longrightarrow (M,G)$, one can always assume that the identity map on $M$ gives equivariant $\epsilon_i$-approximations for some $\epsilon_i\to 0$. We endow $\mathrm{Isom}(M)$ with a natural bi-invariant metric $d$ from its action on $M$:
$$d_G(g_1,g_2)=\max_{x\in M}d_M(g_1\cdot x,g_2\cdot x).$$
Then $H_i$ converges to the limit $G$ with respect to the Hausdorff distance induced by $(\mathrm{Isom}(M),d)$.

In our proof of Proposition \ref{stability}, we use the following results.

\begin{prop}\cite{GK}\label{com_1}
	Let $\mu_1,\mu_2:H\to G$ be two homomorphisms of compact Lie group $H$ into the Lie group $G$ with a bi-invariant Riemannian metric. There exists $\epsilon(G)>0$ such that if $d(\mu_1(h),\mu_2(h))\le\epsilon(G)$ for all $h\in H$, then the subgroups $\mu_1(H)$ and $\mu_2(H)$ are conjugate in $G$.
\end{prop}

\begin{prop}\cite{MRW}\label{com_2}
	Let $G$ be a Lie group with left-invariant Riemannian metric. Then there exists a constant $\epsilon(G)>0$ such that if $\phi:H\to G$ is a map from a Lie group $H$ to $G$ such that
	$$d(\phi(h_1h_2),\phi(h_1)\phi(h_2))\le\epsilon<\epsilon(G)$$
	for all $h_1,h_2\in H$, then there is a Lie group homomorphism $\bar{\phi}:H\to G$ with $d(\bar{\phi}(h),\phi(h))\le 2\epsilon$ for all $h\in H$.
\end{prop}

We call such a map $\phi:H\to G$ with
$$d(\phi(h_1h_2),\phi(h_1)\phi(h_2))\le\epsilon$$
an $\epsilon$-\textit{homomorphism}. In practice, we may start with some bi-invariant distance function $d$ on $G$. We can equip $G$ with a bi-invariant Riemannian metric $d_0$ (we can do this because $G$ is compact). Then there is some constant $C\ge 1$ such that $C^{-1}d_0 \le d_G \le Cd_0$. With this observation, for a sequence of $\epsilon_i$-homomorphisms with respect to $d$, it is a sequence of $C\epsilon_i$-homomorphisms with respect to $d_0$. Therefore, we can still apply Proposition \ref{com_2} for $i$ large.

\begin{proof}[Proof of Proposition \ref{stability}]
	Let $\{H_i\}$ be a sequence of group actions on $M$ such that $(M,H_i)\overset{GH}\longrightarrow(M,G)$. We show that if $\dim(H_i)\ge \dim(G)$, then $(M,H_i)$ is conjugate to $(M,G)$ by an isometry for all $i$ large.
	
	As pointed out, for $(M,H_i)\overset{GH}\longrightarrow(M,G)$, it is equivalent to consider the Hausdorff convergence $H_i\overset{H}\to G$ in $(\mathrm{Isom}(M),d)$, where $d$ is given by
	$$d(g_1,g_2)=\max_{x\in M} d_{M}(g_1 \cdot x,g_2\cdot x).$$
	We know that there is $\epsilon_i\to 0$ such that $d_H(H_i,G)\le \epsilon_i$. For each $h\in H_i$, we choose $\phi_i(h)$ as an element in $G$ that is $\epsilon_i$-close to $h$. This defines a map
	$$\phi_i: H_i\to G.$$
	It is straight-forward to check that $\phi_i$ is a $3\epsilon_i$-homomorphism with respect the metric $d|_G$:
	\begin{align*}
		&\ d(\phi_i(h_1h_2),\phi_i(h_1)\phi_i(h_2))\\
		\le	&\ d(\phi_i(h_1h_2),h_1h_2)+d(h_1h_2,h_1\phi_i(h_2))+d(h_1\phi_i(h_2),\phi_i(h_1)\phi_i(h_2))\\
		\le&\ 3\epsilon_i
	\end{align*}
    for any $h_1,h_2\in H_i$.
	Apply Proposition \ref{com_2}, we obtain a sequence of Lie group homomorphisms:
	$$\bar{\phi_i}: H_i\to G.$$
	
	\textbf{Claim:} $\bar{\phi}_i$ is a Lie group isomorphism for all $i$ large. We first show that $\bar{\phi}_i$ is injective. Suppose that $\ker(\bar{\phi}_i)\not=\{e\}$, then we have a sequence of non-trivial subgroups converging to $\{e\}$:
	$$\ker(\bar{\phi}_i)\overset{H}\to\{e\}.$$
	However, there exists $\delta>0$ such that any non-trivial subgroup of $\mathrm{Isom}(M)$ has displacement at least $\delta$ on $M$. This is because $\mathrm{Isom}(M)$ is a Lie group, which can not have arbitrarily small non-trivial subgroups. Thus $\ker(\bar{\phi}_i)=\{e\}$ for all $i$ large. Recall the assumption that $\dim(H_i)\ge \dim(G)$. Since $\bar{\phi}_i$ is injective, we must have $\dim(H_i)=\dim(G)$. Also note that the image $\bar{\phi}_i(H_i)$ is $C\epsilon_i$-dense in $G$, thus $\bar{\phi_i}$ must be surjective for $i$ large.
	
	Now $H_i$ has two embeddings into $\mathrm{Isom}(M)$:
	$$\iota_i:H_i\to \mathrm{Isom}(M),\quad \bar{\phi}_i: H_i\to G\subseteq \mathrm{Isom}(M),$$
	where $\iota_i$ is the inclusion map. Note that $$d(h,\bar{\phi}_i(h))\le C\epsilon_i\to 0$$
	for all $h\in H_i$ and some constant $C$. By Proposition \ref{com_1}, we conclude that for $i$ large $G=\bar{\phi}_i(H_i)$ is conjugate to $H_i$ as subgroups in $\mathrm{Isom}(M)$. In other words, there is some isometry $g_i\in \mathrm{Isom}(M)$ such that $g_i^{-1}Gg_i=H_i$. This shows that $(M,H_i)$ and $(M,G)$ are conjugate through isometry $g_i$ for $i$ large.
\end{proof}

\begin{rem}\label{rem_stable}
	The proof of Proposition \ref{stability} can be immediately extended to any compact metric space $(X,d)$ whose isometry group $\mathrm{Isom}(X)$ is a Lie group. In particular, for a metric cone $\mathbb{R}^k\times C(Z)\in \mathcal{M}(n,0)$, where $\mathrm{diam}(Z)<\pi$, since $\mathrm{Isom}(Z)$ a Lie group \cite{CC3,CN1}, we can extend Proposition \ref{stability} to such a space $Z$.  
\end{rem}

\section{Proof of Theorem B}

Without mentioning, we always assume that groups in this section are abelian. We prove Theorem B in two steps. First we show that for all $$(\widetilde{Y},\tilde{y},G)=(\mathbb{R}^k\times C(Z),(0,z),G)\in\Omega(\widetilde{M},\Gamma),$$ the isotropy subgroup of $p(G)$ at $0$ is independent of $(\widetilde{Y},\tilde{y},G)$, and $(\mathbb{R}^k,0,p(G))$ satisfies property \textit{(P)} (see Definition \ref{linear} below). Secondly, we prove the non-compact factor in $p(G)$ is also independent of $(\widetilde{Y},\tilde{y},G)$. The proof of each step shares the same structure as Proposition \ref{not_type_0}: we show that there exists a gap between two certain classes of group actions, then choose a critical rescaling to derive a desired contradiction in the corresponding limit space.

Recall that once we specify a point in $\mathbb{R}^k$ as the origin $0$, then every element in $\mathrm{Isom}(\mathbb{R}^k)=\mathbb{R}^k\rtimes O(k)$ can be written as $(A,x)$, where $A\in O(k)$ fixing $0$ and $x\in\mathbb{R}^k$. For convenience, we introduce a definition.

\begin{defn}\label{def_linear}
	Let $(\mathbb{R}^k,0,G)$ be the $k$ dimensional Euclidean space with an isometric abelian $G$-action. We say that $(\mathbb{R}^k,0,G)$ satisfies \textit{property (P)}, if\\ \textit{(P)} for any element $(A,v)\in G$, $(A,0)$ is also an element of $G$.
\end{defn}

Property \textit{(P)} has the following consequence.

\begin{lem}\label{linear}
	If $(\mathbb{R}^k,0,G)$ satisfies \textit{property (P)}, then\\
(1) any compact subgroup of $G$ fixes $0$;\\
(2) $G$ admits decomposition $G=\mathrm{Iso}_0 G\times \mathbb{R}^{l}\times \mathbb{Z}^{m}$, and any element in the subgroup $\{e\}\times \mathbb{R}^{l}\times \mathbb{Z}^{m}$ is a translation, where $\mathrm{Iso}_0 G$ is the isotropy subgroup of $G$ at $0$.
\end{lem}

\begin{proof}
	(1) Let $K$ be any compact subgroup of $G$. Suppose that $K$ does not fix $0$. Then there is $g=(A,v)\in K$ such that
	$$0\not=g\cdot 0=(A,v)\cdot 0=x.$$
	By assumption, $(A,0)\in G$. Hence $(A,x)\cdot(A^{-1},0)=(I,v)$ is also an element of $G$. Because $G$ is abelian, $(A,0)$ and $(I,v)$ commutes. This implies that $A\cdot v=v$, and thus $(A,v)^k=(A^k,kv)$ for any integer $k$. We see that the subgroup generated by $(A,v)$ can not be contained in any compact group, a contradiction.\\
	(2) This follows from (1) and the structure of abelian Lie groups.
\end{proof}

It is clear that (2) in Lemma \ref{linear} is equivalent to property \textit{(P)}.

\begin{rem}\label{rem_linear_1}
	Let ${M}$ be an open $n$-manifold of $\mathrm{Ric}\ge 0$. Suppose that $$(\widetilde{Y},\tilde{y},G)=(\mathbb{R}^k\times C(Z),(0,z),G)\in \Omega(\widetilde{M},\Gamma)$$
	is a metric cone with isometric $G$-action, where $C(Z)$ has vertex $z$ and $\mathrm{diam}(Z)<\pi$. We do not know any example so that $(\mathbb{R}^k,0,p(G))$ does not satisfy property \textit{(P)}. However, we can always find ones with property \textit{(P)} in $\Omega(\widetilde{M},\Gamma)$ by passing to the equivariant tangent cone of $(\widetilde{Y},\tilde{y},G)$ at $\tilde{y}$, or at infinity ($j\to\infty$):
	$$(j\widetilde{Y},\tilde{y},G)\overset{GH}\longrightarrow(\widetilde{Y},\tilde{y},G_{\tilde{y}}),$$
	$$(j^{-1}\widetilde{Y},\tilde{y},G)\overset{GH}\longrightarrow(\widetilde{Y},\tilde{y},G_\infty).$$
	It is clear that both $(\mathbb{R}^k,0,p(G_{\tilde{y}}))$ and $(\mathbb{R}^k,0,p(G_\infty))$ satisfy property \textit{(P)}, because they satisfy (2) in Lemma \ref{linear} (with no $\mathbb{Z}$-factors).
\end{rem}

\begin{rem}\label{rem_linear_2}
If $(\mathbb{R}^k,0,G)$ does not satisfy property \textit{(P)}, then there is an element $(A,v)\in G$, but $(A,0)\not\in G$. After blowing down
$$(j^{-1}\mathbb{R}^k,0,G,(A,v))\overset{GH}\longrightarrow(\mathbb{R}^k,0,G_\infty,(A,0)).$$
Thus $(A,0)\in G_\infty$. Note that $\mathrm{Iso}_0p(G)$ is preserved as a subgroup of $\mathrm{Iso}_0p(G_\infty)$. Hence $\mathrm{Iso}_0p(G)$ is a proper subgroup of $\mathrm{Iso}_0p(G_\infty)$.
\end{rem}

We restate Theorem B in terms of Definition \ref{def_linear}.

\begin{thm}\label{omega}
	Let $M$ be an open $n$-manifold with abelian fundamental group and $\mathrm{Ric}\ge 0$, whose universal cover $\widetilde{M}$ is $k$-Euclidean at infinity. Then there exist a closed abelian subgroup $K$ of $O(k)$ and an integer $l\in[0,k]$ such that for any space $(\widetilde{Y},\tilde{y},G)\in \Omega(\widetilde{M},\Gamma)$, $(\mathbb{R}^k,0,p(G))$ satisfies property \textit{(P)} and $p(G)=K\times \mathbb{R}^l$.
\end{thm}

For convenience, we introduce a definition.

\begin{defn}
	Let $(Y_j,y_j)$ be a metric space with isometric Lie group $G_j$-action ($j=1,2$). We say that $(Y_1,y_1,G_1)$ is equivalent to $(Y_2,y_2,G_2)$, if
	$$d_{GH}((Y_1,y_1,G_1),(Y_2,y_2,G_2))=0;$$
	or equivalently, there is an isometry $F:Y_1\to Y_2$ with $F(y_1)=y_2$, and a Lie group isomorphism $\psi:G_1\to G_2$ such that $F(g_1\cdot x_1)=\psi(g_1)\cdot F(x_1)$ for any $g_1\in G$ and $x_1\in Y_1$.
\end{defn}

We first establish a gap phenomenon between two classes of actions with property \textit{(P)} but different projected isotropy groups.

\begin{lem}\label{gap_isotropy}
	Let $M$ be an open $n$-manifold of $\mathrm{Ric}\ge 0$, whose universal cover is $k$-Euclidean at infinity. Let $K$ be an isometric action on $\mathbb{R}^k$ fixing $0$. Then there exists $\epsilon>0$, depending on $M$ and $K$-action, such that the following holds.	
	
	For any two spaces $(\widetilde{Y}_j,\tilde{y}_j,G_j)\in \Omega(\widetilde{M},\Gamma)$ $(j=1,2)$ satisfying\\
	(1) $(\mathbb{R}^k,0,p(G_j))$ satisfies property \textit{(P)} ($j=1,2$),\\
	(2) $(\mathbb{R}^k,0,\mathrm{Iso}_0 p(G_1))$ is equivalent to $(\mathbb{R}^k,0,K)$,\\
	(3) $\dim(\mathrm{Iso}_0 p(G_2))\ge\dim (K)$ and $(\mathbb{R}^k,0,\mathrm{Iso}_0 p(G_2))$ is not equivalent to $(\mathbb{R}^k,0,K)$,\\
	then
	$$d_{GH}((\widetilde{Y}_1,\tilde{y}_1,G_1),(\widetilde{Y}_2,\tilde{y}_2,G_2))\ge\epsilon.$$
\end{lem}

\begin{proof}
	Suppose that there are two sequences in $\Omega(M)$: $\{(\widetilde{Y}_{ij},\tilde{y}_{ij},G_{ij})\}_i$ $(j=1,2)$ such that for all $i$,\\
	(1) $(\mathbb{R}^k,0,p(G_{ij}))$ satisfies property \textit{(P)} ($j=1,2$);\\
	(2) $(\mathbb{R}^k,0,\mathrm{Iso}_0p(G_{i1}))$ is equivalent to $(\mathbb{R}^k,0,K)$;\\
	(3) $\dim(K_i)\ge\dim (K)$ and $(\mathbb{R}^k,0,K_i)$ is not equivalent to $(\mathbb{R}^k,0,K)$, where $K_i=\mathrm{Iso}_0 p(G_{i2})$;\\
	(4) $d_{GH}((\widetilde{Y}_{i1},\tilde{y}_{i1},G_{i1}),(\widetilde{Y}_{i2},\tilde{y}_{i2},G_{i2}))\to 0$ as $i\to\infty$.\\
	After passing to some subsequences, this gives convergence
	$$(\widetilde{Y}_{i1},\tilde{y}_{i1},G_{i1})\overset{GH}\longrightarrow(\widetilde{Y}_\infty,\tilde{y}_\infty,G_\infty),$$
	$$(\widetilde{Y}_{i2},\tilde{y}_{i2},G_{i2})\overset{GH}\longrightarrow(\widetilde{Y}_\infty,\tilde{y}_\infty,G_\infty),$$
	with $(\widetilde{Y}_\infty,\tilde{y}_\infty)=(\mathbb{R}^k\times C(Z_\infty),(0,z_\infty))$, where $\mathrm{diam}(Z_\infty)<\pi$ and $z_\infty$ is the vertex of $C(Z_\infty)$ (Lemma \ref{gap_higher_sym}). Consequently,
	$$(\mathbb{R}^k,0,p(G_{i1}))\overset{GH}\longrightarrow(\mathbb{R}^k,0,p(G_\infty)),$$
	$$(\mathbb{R}^k,0,p(G_{i2}))\overset{GH}\longrightarrow(\mathbb{R}^k,0,p(G_\infty)).$$
	Because each $(\mathbb{R}^k,0,p(G_{ij}))$ satisfies property \textit{(P)} for all $i$ and $j$, we conclude that
	$$(S^{k-1},K)\overset{GH}\longrightarrow(S^{k-1},K_\infty),$$
	$$(S^{k-1},K_i)\overset{GH}\longrightarrow(S^{k-1},K_\infty),$$
	where $S^{k-1}$ is the unit sphere in $\mathbb{R}^{k}$ and $K_\infty=\mathrm{Iso}_0p(G_{\infty})$. It is obvious that $(S^{k-1},0,K_\infty)$ is equivalent to  $(S^{k-1},0,K)$. By Proposition \ref{stability}, for all $i$ sufficiently large, either $\dim(K_i)<\dim(K)$ or $(S^{k-1},0,K_i)$ is equivalent to $(S^{k-1},0,K)$. This contradicts the hypothesis (3) on $K_i$.
\end{proof}

\begin{lem}\label{omega_linear}
	Let $M$ be an open $n$-manifold of $\mathrm{Ric}\ge 0$ and abelian fundamental group. Suppose that $\widetilde{M}$ is $k$-Euclidean at infinity. Then for any space $(\widetilde{Y},\tilde{y},G)\in\Omega(\widetilde{M},\Gamma)$, $p(G)$-action on $(\mathbb{R}^k,0,G)$ satisfies property \textit{(P)}. Moreover, $\mathrm{Iso}_0p(G)$ is independent of $(\widetilde{Y},\tilde{y},G)$.
\end{lem}

The key to prove Lemma \ref{omega_linear} is the following lemma.

\begin{lem}\label{omega_isotropy}
	Let $M$ be an open $n$-manifold of $\mathrm{Ric}\ge 0$ and abelian fundamental group. Suppose that $\widetilde{M}$ is $k$-Euclidean at infinity. Then for any two spaces $(\widetilde{Y}_j,\tilde{y}_j,G_j)\in \Omega(\widetilde{M},\Gamma)$ with $(\mathbb{R}^k,0,p(G_j))$ satisfying property \textit{(P)} $(j=1,2)$, $(\mathbb{R}^k,0,\mathrm{Iso}_0 p(G_1))$ must be equivalent to $(\mathbb{R}^k,0,\mathrm{Iso}_0 p(G_2))$.
\end{lem}

We prove Lemma \ref{omega_isotropy} by induction, in terms of the following order on the set of all compact abelian Lie groups.

\begin{defn}
	For a compact Lie group $K$, we define $D(K)=(\dim K,\#K/K_0)$. For two compact Lie groups $K$ and $H$, with $D(K)=(l_1,l_2)$ and $D(H)=(m_1,m_2)$, we say that $D(K)<D(H)$, if $l_1<m_1$, or if $l_1=m_1$ and $l_2<m_2$. We say that $D(K)\le D(H)$, if $D(K)=D(H)$ or $D(K)<D(H)$.
\end{defn}

\begin{proof}[Proof of Lemma \ref{omega_isotropy}]
	We argue by contradiction. Suppose that there are two spaces $(\widetilde{Y}_j,\tilde{y}_j,G_j)\in \Omega(\widetilde{M},\Gamma)$ such that
	$(\mathbb{R}^k,0,p(G_j))$ satisfies property \textit{(P)} $(j=1,2)$, and
	$(\mathbb{R}^k,0,\mathrm{Iso}_0 p(G_1))$ is not equivalent to $(\mathbb{R}^k,0,\mathrm{Iso}_0 p(G_2))$. We derive a contradiction by the critical rescaling argument and Lemma \ref{gap_isotropy}.

 We argue this by induction on $\min\{D(K_1),D(K_2)\}$, where $K_j=\mathrm{Iso}_0 p(G_j)$ $(j=1,2)$. Without lose of generality, we assume that $D(K_1)\le D(K_2)$. Assuming that the above can not happen when $D(K_1)<(m_1,m_2)$, we will derive a contradiction for $D(K_1)=(m_1,m_2)$.
	
	Let $r_i\to\infty$ and $s_i\to\infty$ be two sequences such that
	$$(r^{-1}_i\widetilde{M},\tilde{x},\Gamma)\overset{GH}\longrightarrow(\widetilde{Y}_1,\tilde{y}_1,G_1),$$
	$$(s^{-1}_i\widetilde{M},\tilde{x},\Gamma)\overset{GH}\longrightarrow(\widetilde{Y}_2,\tilde{y}_2,G_2),$$
	and $t_i:=(s^{-1}_i)/(r^{-1}_i)\to\infty$. Put $(N_i,q_i,\Gamma_i)=(r^{-1}_i\widetilde{M},\tilde{x},\Gamma)$, then we have
	$$(N_i,q_i,\Gamma_i)\overset{GH}\longrightarrow(\widetilde{Y}_1,\tilde{y}_1,G_1),$$
	$$(t_iN_i,q_i,\Gamma_i)\overset{GH}\longrightarrow(\widetilde{Y}_2,\tilde{y}_2,G_2).$$
	We know that $(\mathbb{R}^k,0,p(G_j))$ satisfies property \textit{(P)} $(j=1,2)$, $D(K_1)=(m_1,m_2)\le D(K_2)$, and $(\mathbb{R}^k,0,K_1)$ is not equivalent to $(\mathbb{R}^k,0,K_2)$.
	
    For each $i$, we define a set of scales
	\begin{align*}
		L_i:=\{\  1\le l\le t_i\ |\ & d_{GH}((l{N}_i,q_i,\Gamma_i),({W},{w},H))\le \epsilon/3 \\
		& \text{ for some space $(W,w,H)\in \Omega(\widetilde{M},\Gamma)$}\\
		& \text{ such that $(\mathbb{R}^k,0,p(H))$ satisfies property \textit{(P)};}\\
		& \text{ moreover, } D(\mathrm{Iso}_0 p(H))>(m_1,m_2) \text{, or}\\
		& D(\mathrm{Iso}_0 p(H))=(m_1,m_2) \text{ but } (\mathbb{R}^k,0,\mathrm{Iso}_0 p(H)) \\
		& \text{is not equivalent to } (\mathbb{R}^k,0,K_1)\}.
	\end{align*}
    We choose the above $\epsilon>0$ as follows: by Lemma \ref{gap_isotropy}, there is $\epsilon>0$, depending on ${M}$ and $(\mathbb{R}^k,0,K_1)$ such that for any $(W_j,w_j,H_i)\in \Omega(\widetilde{M},\Gamma)$ $(j=1,2)$ satisfying\\
    (1) $(\mathbb{R}^k,0,p(H_j))$ satisfies property \textit{(P)} $(j=1,2)$,\\
    (2) $(\mathbb{R}^k,0,\mathrm{Iso}_0p(H_1))$ is equivalent to $(\mathbb{R}^k,0,K_1)$,\\
    (3) $d_{GH}((W_1,w_1,H_1),(W_2,w_2,H_2))\le\epsilon$,\\ then $\dim(\mathrm{Iso}_0p(H_2))<\dim(K_1)$, or $(\mathbb{R}^k,0,\mathrm{Iso}_0p(H_2))$ is equivalent to $(\mathbb{R}^k,0,K_1)$.

    Since $t_i\in L_i$ for $i$ large, we choose $l_i\in L_i$ with $\inf L_i\le l_i \le \inf L_i+1/i$.

    \textbf{Claim 1:} $l_i\to\infty$. Suppose that $l_i\to C<\infty$ for some subsequence, then for this subsequence,
    $$(l_iN_i,q_i,\Gamma_i)\overset{GH}\longrightarrow(C\cdot \widetilde{Y}_1,\tilde{y}_1,G_1).$$
    Together with the fact that $l_i\in L_i$, we know that there is some space $(W,w,H)\in\Omega(\widetilde{M},\Gamma)$ with the properties below:\\
    (1) $(\mathbb{R}^k,0,p(H))$ satisfies property \textit{(P)},\\
    (2) $D(\mathrm{Iso}_0 p(H))>(m_1,m_2)$, or $D(\mathrm{Iso}_0 p(H))=(m_1,m_2)$ but $(\mathbb{R}^k,0,\mathrm{Iso}_0 p(H))$ is not equivalent to $(\mathbb{R}^k,0,K_1)$,\\
    (3) $d_{GH}((C\cdot \widetilde{Y}_1,\tilde{y}_1,G_1),(W,w,H))\le\epsilon/2$.\\
    Since $(\mathbb{R}^k,0,p(H))$ satisfies property \textit{(P)}, we see that $(\mathbb{R}^k,0,\mathrm{Iso}_0p(H))$ is equivalent to $(C\cdot \mathbb{R}^k,0,\mathrm{Iso}_0p(H))$. By Lemma \ref{gap_isotropy} the choice of $\epsilon$, we conclude that either $\dim(\mathrm{Iso}_0p(H))<\dim K_1$, or $(\mathbb{R}^k,0,\mathrm{Iso}_0 p(H))$ is equivalent to $(\mathbb{R}^k,0,K_1)$, which is a contradiction to the condition (2) above. We have verified Claim 1.

    Passing to a subsequence if necessary, we have convergence
    $$(l_iN_i,q_i,\Gamma_i)\overset{GH}\longrightarrow(\widetilde{Y}',\tilde{y}',G').$$
    To draw a contradiction, the goal is ruling out all the possibilities of $p(G')$-action.

    \textbf{Claim 2:} $D(K')\ge (m_1,m_2)$, where $K'=\mathrm{Iso}_0p(G')$. If $D(K')<(m_1,m_2)$, we pass to the equivariant tangent cone of $(\widetilde{Y}',\tilde{y}',G')$ at $\tilde{y}'$. In this way, we have $(\widetilde{Y}',\tilde{y}',G'_{\tilde{y}'})$ with $(\mathbb{R}^k,0,p(G'_{\tilde{y}'}))$ satisfying property \textit{(P)} (see Remark \ref{rem_linear_1}). Note that $(\mathbb{R}^k,0,\mathrm{Iso}_0p(G'_{\tilde{y}'}))$ is equivalent to $(\mathbb{R}^k,0,K')$ and $D(K')<(m_1,m_2)$. We know that this can not happen due to the induction assumption.

    \textbf{Claim 3:} $(\mathbb{R}^k,0,p(G'))$ satisfies property \textit{(P)}, and $D(K')=(m_1,m_2)$. In fact, we pass to the equivariant tangent cone of $(\widetilde{Y}',\tilde{y}',G')$ at infinity:
    $$(j^{-1}\widetilde{Y}',\tilde{y}',G')\overset{GH}\longrightarrow(\widetilde{Y}',\tilde{y}',G'_\infty).$$
    For this space, $(\mathbb{R}^k,0,p(G'_\infty))$ satisfies property \textit{(P)} (see Remark \ref{rem_linear_1}). Suppose that Claim 3 fails, then $D(\mathrm{Iso}_0p(G'_\infty))>(m_1,m_2)$ (see Remark \ref{rem_linear_2}). We choose a large integer $J$ such that
    $$d_{GH}((J^{-1}\widetilde{Y}',\tilde{y}',G'),(\widetilde{Y}',\tilde{y}',G'_\infty))\le\epsilon/4.$$
    Hence for all $i$ large, we have
    $$d_{GH}((J^{-1}l_iN_i,q_i,\Gamma_i),(\widetilde{Y}',\tilde{y}',G'_\infty))\le\epsilon/3.$$
    This implies that $l_i/J\in L_i$ for all $i$ large, which is a contradiction to our choice of $l_i$.

    \textbf{Claim 4:} $(\mathbb{R}^k,0,K')$ is equivalent to $(\mathbb{R}^k,0,K_1)$. Suppose not, then we consider the sequence $l_i/2$:
    $$(l_i/2\cdot N_i,q_i,\Gamma_i)\overset{GH}\longrightarrow(1/2 \cdot \widetilde{Y}',\tilde{y}',G').$$
    Since $(\mathbb{R}^k,0,p(G'))$ satisfies property \textit{(P)}, $(1/2\cdot \mathbb{R}^k,0,K')$ is equivalent to $(\mathbb{R}^k,0,K')$, which is not equivalent to $(\mathbb{R}^k,0,K_1)$. This means that $l_i/2\in L_i$ for $i$ large. A contradiction.

    This leads to the ultimate contradiction: Because $l_i\in L_i$, there is some space $(W,w,H)\in\Omega(\widetilde{M},\Gamma)$ satisfying the conditions (1)(2) in the proof of Claim 1, and
    $$d_{GH}((\widetilde{Y}',\tilde{y}',G'),(W,w,H))\le\epsilon/2.$$
    On the other hand, by Claims 3, 4, Lemma \ref{gap_isotropy} and the choice of $\epsilon$, $(W,w,H)$ can not fulfill condition (2) (cf. proof of Claim 1).

    For the remaining base case $D(K_1)=(0,0)$, note that in the above proof, the induction assumption is only used in Claim 2 to conclude $D(K')\ge (m_1,m_2)$. For the base case $(m_1,m_2)=(0,0)$, it is trivial that $D(K')\ge(0,0)$.
\end{proof}

\begin{proof}[Proof of Lemma \ref{omega_linear}]
	With Lemma \ref{omega_isotropy}, it is enough to show that for any space $(\widetilde{Y},\tilde{y},G)\in\Omega(\widetilde{M},\Gamma)$, $(\mathbb{R}^k,0,p(G))$ always satisfies property \textit{(P)}. Suppose the contrary, that is, $(\mathbb{R}^k,0,p(G))$ does not satisfy property \textit{(P)} for some $(\widetilde{Y},\tilde{y},G)$ in $\Omega(\widetilde{M},\Gamma)$. We pass to the equivariant tangent cone of $(\widetilde{Y},\tilde{y},G)$ at $\tilde{y}$ and at infinity respectively (see Remark \ref{rem_linear_1}). We obtain $(\widetilde{Y},\tilde{y},G_{\tilde{y}})$ and $(\widetilde{Y},\tilde{y},G_\infty).$
	For these two spaces, $(\mathbb{R}^k,0,p(G_{\tilde{y}}))$ and $(\mathbb{R}^k,0,p(G_\infty))$ always satisfy property \textit{(P)}. By Lemma \ref{omega_isotropy}, $(\mathbb{R}^k,0,\mathrm{Iso}_0p(G_{\tilde{y}}))$ is equivalent to $(\mathbb{R}^k,0,\mathrm{Iso}_0p(G_\infty))$.

On the other hand, because $(\mathbb{R}^k,0,p(G))$ does not satisfy property \textit{(P)}, $\mathrm{Iso}_0p(G)$ is a proper subgroup of $\mathrm{Iso}_0p(G_\infty)$ (Remark \ref{rem_linear_2}). Since $\mathrm{Iso}_0p(G)=\mathrm{Iso}_0p(G_{\tilde{y}})$, we conclude that $(\mathbb{R}^k,0,\mathrm{Iso}_0p(G_{\tilde{y}}))$ and $(\mathbb{R}^k,0,\mathrm{Iso}_0p(G_\infty))$ can not be equivalent, a contradiction to Lemma \ref{omega_isotropy}.
\end{proof}

Lemmas \ref{linear} and \ref{omega_linear} imply that there exists a closed subgroup $K$ of $O(k)$ such that for any $(\widetilde{Y},\tilde{y},G)\in\Omega(\widetilde{M},\Gamma)$, $(\mathbb{R}^k,0,p(G))$ satisfies property \textit{(P)}, and $p(G)=K\times \mathbb{R}^{l}\times \mathbb{Z}^{m}$. To finish the proof of Theorem \ref{omega}, we need to show that $l$ is independent of $(\widetilde{Y},\tilde{y},G)$ and $m$ is always $0$.

We prove the following gap lemma on the non-compact factor of $p(G)$, which does not require Lemma \ref{gap_isotropy}.

\begin{lem}\label{gap_trans}
	Let $M$ be an open $n$-manifold of $\mathrm{Ric}\ge 0$. Suppose that $\widetilde{M}$ is $k$-Euclidean at infinity. Then there exists $\epsilon(M)>0$ such that the following holds.
	
	For any two spaces $(\widetilde{Y}_j,\tilde{y}_i,G_j)\in \Omega(\widetilde{M},\Gamma)$ $(j=1,2)$ satisfying\\
	(1) $(\mathbb{R}^k,0,p(G_j))$ satisfies property \textit{(P)} ($j=1,2$),\\
	(2) $p(G_1)=\mathrm{Iso}_0p(G_1)\times\mathbb{R}^l$ (cf. Lemma \ref{linear}),\\
	(3) $p(G_2)$ contains $\mathbb{R}^l\times \mathbb{Z}$ as a closed subgroup; for this extra $\mathbb{Z}$ subgroup, it has generator $\gamma$ with $d_{\mathbb{R}^k}(\gamma\cdot0,0)\le 1$.\\
	Then
	$$d_{GH}((\widetilde{Y}_1,\tilde{y}_1,G_1),(\widetilde{Y}_2,\tilde{y}_2,G_2))\ge\epsilon(M).$$
\end{lem}

\begin{proof}
	We argue by contradiction. Suppose that there are two sequences of spaces in $\Omega(\widetilde{M},\Gamma)$: $\{(\widetilde{Y}_{ij},\tilde{y}_{ij},G_{ij})\}_i$ $(j=1,2)$ such that\\
	(1) $(\mathbb{R}^k,0,p(G_{ij}))$ satisfies property \textit{(P)} ($j=1,2$);\\
	(2) $p(G_{i1})=K_{i1}\times \mathbb{R}^l$, where $K_{i1}=\mathrm{Iso}_0p(G_{i1})$;\\
	(3) $p(G_{i2})$ contains $\mathbb{R}^l\times \mathbb{Z}$ as a closed subgroup; for this extra $\mathbb{Z}$ subgroup, it has generator $\gamma_i$ with $d_{\mathbb{R}^k}(\gamma_i\cdot0,0)\le 1$;\\
	(4) $d_{GH}((\widetilde{Y}_{i1},\tilde{y}_{i1},G_{i1}),(\widetilde{Y}_{i2},\tilde{y}_{i2},G_{i2}))\to 0$ as $i\to\infty$.\\
	This gives the convergence
	$$(\widetilde{Y}_{i1},\tilde{y}_{i1},G_{i1})\overset{GH}\longrightarrow(\widetilde{Y}_\infty,\tilde{y}_\infty,G_\infty),$$
	$$(\widetilde{Y}_{i2},\tilde{y}_{i2},G_{i2})\overset{GH}\longrightarrow(\widetilde{Y}_\infty,\tilde{y}_\infty,G_\infty);$$
	thus
	$$(\mathbb{R}^k,0,p(G_{i1}))\overset{GH}\longrightarrow(\mathbb{R}^k,0,p(G_\infty)),$$
	$$(\mathbb{R}^k,0,p(G_{i2}))\overset{GH}\longrightarrow(\mathbb{R}^k,0,p(G_\infty)).$$
	Since $p(G_{i1})=K_{i1}\times \mathbb{R}^l$ and $(\mathbb{R}^k,0,p(G_{i1}))$ satisfies property \textit{(P)}, we conclude that  $(\mathbb{R}^k,0,p(G_\infty))$ also satisfies property \textit{(P)}, and $p(G_\infty)=K_\infty\times\mathbb{R}^l$ with $K_\infty$ fixing $0$. On the other hand, by hypothesis (3), $p(G_{i2})$ contains a proper closed subgroup $H_i=K_{i2}\times \mathbb{R}^l$ with $K_{i2}=\mathrm{Iso}_0p(G_{i2})$. Moreover, there is some element $\alpha_i\in p(G_{i2})$ outside $H_i$ such that $d(H_i\cdot 0,\alpha_i\cdot 0)\in (1,3)$. This yields
	$$(\mathbb{R}^k,0,H_i,\alpha_i)\overset{GH}\longrightarrow(\mathbb{R}^k,0,H_\infty,\alpha_\infty),$$
	where $\alpha_\infty$ is outside $H_\infty$ with $d(H_\infty\cdot 0,\alpha_\infty\cdot0)\in (1,3)$. Therefore, $p(G_\infty)$ also contains $\mathbb{R}^l\times \mathbb{Z}$ as a closed subgroup, a contradiction.
\end{proof}

\begin{proof}[Proof of Theorem \ref{omega}]
	By Lemmas \ref{linear} and \ref{omega_linear}, for any space $(\widetilde{Y},\tilde{y},G)\in\Omega(\widetilde{M},\Gamma)$, $(\mathbb{R}^k,0,p(G))$ always satisfies property \textit{(P)}, and $p(G)=K\times \mathbb{R}^{l}\times \mathbb{Z}^{m}$, where $K$ is a closed subgroup of $O(k)$ independent of $(\widetilde{Y},\tilde{y},G)$. It remains to show that $l$ is independent of $(\widetilde{Y},\tilde{y},G)$ and $m$ is always $0$.
	
	We argue by contradiction with the critical rescaling argument and Lemma \ref{gap_trans}. By passing to the tangent cones at the base point of spaces in $\Omega(\widetilde{M},\Gamma)$, we can choose contradictory two spaces $(\widetilde{Y}_j,\tilde{y}_j,G_j)\in\Omega(\widetilde{M},\Gamma)$ $(j=1,2)$ with $p(G_1)=K\times \mathbb{R}^l$ and $p(G_2)$ containing $\mathbb{R}^l\times \mathbb{Z}$ as a closed subgroup. We rule out this by induction on $l$. Assuming that this can not happen for $p(G_1)$ has non-compact factor with dimension $0,..,l-1$, we prove the case $p(G_1)=K\times \mathbb{R}^l$.
	
	Let $r_i\to\infty$ and $s_i\to\infty$ be two sequences such that
	$$(r^{-1}_i\widetilde{M},\tilde{x},\Gamma)\overset{GH}\longrightarrow(\widetilde{Y}_1,\tilde{y}_1,G_1),$$
	$$(s^{-1}_i\widetilde{M},\tilde{x},\Gamma)\overset{GH}\longrightarrow(\widetilde{Y}_2,\tilde{y}_2,G_2),$$
	and $t_i:=(s^{-1}_i)/(r^{-1}_i)\to\infty$. Rescale $s_i^{-1}$ down by a constant if necessary, we assume that the extra $\mathbb{Z}$ subgroup in $p(G_2)$ has generator $\gamma$ with $d_{\mathbb{R}^k}(\gamma\cdot 0,0)\le 1$. We put $(N_i,q_i,\Gamma_i)=(r^{-1}_i\widetilde{M},\tilde{x},\Gamma)$, then
	$$(N_i,q_i,\Gamma_i)\overset{GH}\longrightarrow(\widetilde{Y}_1,\tilde{y}_1,G_1),$$
	$$(t_iN_i,q_i,\Gamma_i)\overset{GH}\longrightarrow(\widetilde{Y}_2,\tilde{y}_2,G_2).$$
	
	For each $i$, we define
	\begin{align*}
		L_i:=\{\  1\le l\le t_i\ |\ & d_{GH}((l{N}_i,q_i,\Gamma_i),({W},{w},H))\le \epsilon/3 \\
		& \text{ for some space $(W,w,H)\in \Omega(\widetilde{M},\Gamma)$ such that}\\
		& \text{ $p(H)$ contains $\mathbb{R}^l\times \mathbb{Z}$ as a closed subgroup;}\\
		& \text{ moreover, this extra $\mathbb{Z}$-subgroup }\\
		& \text{ has generator $h$ with } d_{\mathbb{R}^k}(h\cdot 0,0)\le 1.\}
	\end{align*}
    In the above definition, we choose $\epsilon=\epsilon({M})>0$ as the constant in Lemma \ref{gap_trans}. $t_i\in L_i$ for $i$ large. We choose $l_i\in L_i$ with $\inf L_i\le l_i \le \inf L_i+1/i$.

    \textbf{Claim 1:} $l_i\to\infty$. If $l_i\to C$, then
    $$(l_iN_i,q_i,\Gamma_i)\overset{GH}\longrightarrow(C\cdot\widetilde{Y}_1,\tilde{y}_1,G_1).$$
    The projection to Euclidean factor $(C\cdot\mathbb{R}^k,0,p(G_1))$ is equivalent to $(\mathbb{R}^k,0,p(G'))$, because the later one satisfies property \textit{(P)} and $p(G')=K\times \mathbb{R}^l$. Since $l_i\in L_i$,
    $$d_{GH}((C\cdot\widetilde{Y}_1,\tilde{y}_1,G_1),(W,w,H))\le\epsilon/2$$
    for some $(W,w,H)\in \Omega(\widetilde{M},\Gamma)$ with $p(H)$ containing $\mathbb{R}^l\times \mathbb{Z}$ as a closed subgroup. Moreover, the extra $\mathbb{Z}$-subgroup has generator $h$ with $d_{\mathbb{R}^k}(h\cdot 0,0)\le 1$. This is a contradiction to our choice of $\epsilon$ and Lemma \ref{gap_trans}.

    Next we consider convergence
    $$(l_iN_i,q_i,\Gamma_i)\overset{GH}\longrightarrow(\widetilde{Y}',\tilde{y}',G').$$

    \textbf{Claim 2:} $p(G')=K\times \mathbb{R}^l$. Indeed, because $(\mathbb{R}^k,0,p(G'))$ satisfies property \textit{(P)}, we can write $p(G')=K\times\mathbb{R}^{l'}\times \mathbb{Z}^{m'}$. We can also assume that $l'\ge l$ due to induction assumption. If $l'>l$ or $m'\not= 0$, then $p(G')$ contains $\mathbb{R}^l\times \mathbb{Z}$ as a closed subgroup. Consequently, $l_i/d\in L_i$ for some constant $d\ge 2$, which contradicts our choice of $l_i$. Hence Claim 2 holds.

    We derive the desired contradiction: $l_i\in L_i$ so $$d_{GH}((\widetilde{Y}',\tilde{y}',G'),(W,w,H))\le\epsilon/2$$
    for some space $(W,w,H)\in \Omega(\widetilde{M},\Gamma)$, where $p(H)$ contains $\mathbb{R}^l\times \mathbb{Z}$ as a closed subgroup, and the extra $\mathbb{Z}$-subgroup has generator $h$ with $d_{\mathbb{R}^k}(h\cdot 0,0)\le 1$. A contradiction to Lemma \ref{gap_trans}.

    For the remaining base case $p(G_1)=K$ ($l=0$), the above proof also goes through. Indeed, note that we use the induction assumption to conclude that $l'\ge l$ in Claim 2. For $l=0$, $l'\ge 0$ holds trivially.
\end{proof}

We conclude this paper with remarks on extensions of Theorems A and B. As indicated in the introduction, we do not pursue these extensions because we do not use them in this paper.

\begin{rem}\label{rem_non_abelian}
	With some mild additional arguments, one can generalize Theorem B to any nilpotent fundamental groups:
	
	\textit{Let $M$ be an open $n$-manifold with nilpotent fundamental group $\mathrm{Ric}\ge 0$, whose universal cover $\widetilde{M}$ is $k$-Euclidean at infinity. Then there exist a closed nilpotent subgroup $K$ of $O(k)$ and an integer $l\in[0,k]$ such that for any space $(\widetilde{Y},\tilde{y},G)\in \Omega(\widetilde{M},\Gamma)$, $(\mathbb{R}^k,0,p(G))$ satisfies property (P) and $p(G)=K\times \mathbb{R}^l$.}

\end{rem}

\begin{rem}
    The proof of Theorem B does not rely on the facts that $\widetilde{M}$ is simply connected and $\pi_1(M,x)$-action is free. Thus Theorem B can be generalized to any isometric nilpotent $G$-action on $M$, if $M$ has $\mathrm{Ric}_M\ge 0$ and is $k$-Euclidean at infinity.
\end{rem}

\begin{rem}\label{rem_fix}
    Theorems A and B holds remains true if one replace the $k$-Euclidean at infinity condition by the following: there is $k$ such that any tangent cone of $\widetilde{M}$ at infinity splits as $(\mathbb{R}^k\times X,(0,x))$, where $(X,x)$ satisfies\\
	(1) $X$ has no lines,\\
	(2) any isometry of $X$ fixes $x$.\\
	With this assumption, tangent cones of $\widetilde{M}$ at infinity may not be metric cones nor be polar spaces. Nevertheless, we still have the desired properties on $\mathrm{Isom}(\widetilde{Y})$ for any $\widetilde{Y}=\mathbb{R}^k\times X\in \Omega(\widetilde{M})$ (cf. Propositions \ref{cone_spit}, \ref{isom_gp_split} and Remark \ref{rem_orb}):\\
	(1) $\mathrm{Isom}(\mathbb{R}^k\times X)=\mathrm{Isom}(\mathbb{R}^k)\times\mathrm{Isom}(X)$,\\
	(2) $g\cdot (v,x)=(p(g)\cdot v,x)$ for any $g\in \mathrm{Isom}(\widetilde{Y})$ and any $v\in\mathbb{R}^k$, where $p:\mathrm{Isom}(\widetilde{Y})\to\mathrm{Isom}(\mathbb{R}^k)$ is the natural projection.\\
These properties are all that we required in our proof of Theorems A and B. 
\end{rem}

\begin{rem}
	If ${M}$ has the unique tangent cone at infinity as a metric cone $(\mathbb{R}^k\times C(Z),(0,z))$, and $G$ is a nilpotent group acting as isometries on $M$, then equivariant stability at infinity holds on the entire cone, that is, $\Omega(M,G)$ consists of a single element. More precisely, there exist a closed nilpotent subgroup $K$ and an integer $l\in[0,k]$ such that $(\mathbb{R}^k\times C(Z),(0,z),\mathbb{R}^l\times K)$ is the only element in $\Omega(M,G)$, where $K$ fixing $(0,z)$ and the subgroup $\{e\}\times \mathbb{R}^l$ acting as translations in the $\mathbb{R}^k$-factor.
	
	This can be proved with Remark \ref{rem_stable} and a similar argument in Section 4.
\end{rem}

\Addresses

\end{document}